\title{Pwer sum of element ordrs}
\author{hiranya.dey }
\date{August 2022}

\documentclass[preprint,12pt]{elsarticle}

\usepackage{amssymb}
\usepackage{amsthm}
\usepackage{graphics}
\usepackage{epsfig}
\usepackage{amssymb}
\usepackage{graphicx}
\usepackage{subfigure}
\usepackage{color}
\usepackage{tikz}
\usepackage{tikz,pgfplots}
\pgfplotsset{compat=1.12,axis lines=center}
\usepgfplotslibrary{fillbetween}
\usetikzlibrary{patterns}
\usepackage{hyperref}
\usepackage{a4wide}
\usepackage{amsmath}
\usepackage{amsfonts}
\usepackage{enumerate}
\usepackage{multicol}

\newcommand{\Z}{\mathbb{Z}}

\hypersetup{
    colorlinks=true,
    linkcolor=blue,
    filecolor=dviolet,      
    urlcolor=blue,
} 

\newtheorem*{theoremaux}{Theorem \theoremauxnum}
\gdef\theoremauxnum{1}

\newtheorem{lemma}{\bf Lemma}[section]
 \newtheorem{example}{\bf Example}[section]
\newtheorem{theorem}{\bf Theorem}[section]

\newtheorem{corollary}[lemma]{\bf Corollary}

\newtheorem{remark}{\bf Remark}[section]

\journal{~}

\begin{document}

\begin{frontmatter}



\title{An exact upper bound for the sum of powers of element
orders in non-cyclic finite groups}



\author{Hiranya Kishore Dey}
\ead{hiranya.dey@gmail.com} 
\address{Department of Mathematics,\\ 
Indian Institute of Science, Bangalore \\ India} 

\author{Archita Mondal}
\ead{architamondal40@gmail.com} 
\address{Department of Mathematics,\\ 
Indian Institute of Technology, Bombay 
\\ India}

%


\begin{abstract}

\noindent 
For a finite group $G$,  let $\psi(G)$ denote the sum of element orders of $G$. This function was introduced by Amiri, Amiri, and Isaacs in 2009 and they proved that for any finite group $G$ of order $n$, $\psi(G)$ is maximum if and only if $G \simeq \mathbb{Z}_n$ where  $\mathbb{Z}_n$ denotes the cyclic group of order $n$.  Furthermore, Herzog, Longobardi, and Maj in 2018 proved that if $G$ is non-cyclic, $\psi(G) \leq \frac{7}{11} \psi(\mathbb{Z}_n)$. 
Amiri and Amiri in 2014 introduced the function $\psi_k(G)$ which is defined as the sum of the $k$-th powers of element orders of $G$ and they showed that for every positive integer $k$,
$\psi_k(G)$ is also maximum if and only if $G$ is cyclic.

\medskip

\noindent 
In this paper, we have been able to prove that if $G$ is a non-cyclic group of order $n$, then $\psi_k(G) \leq \frac{1+3.2^k}{1+2.4^k+2^k} \psi_k(\mathbb{Z}_n)$. Setting $k=1$ in our result, we immediately get the result of Herzog et al. as a simple corollary. Besides, a recursive formula for $\psi_k(G)$ is also obtained for finite abelian $p$-groups $G$, using which one can explicitly find out the exact value of $\psi_k(G)$ for finite abelian groups $G$. 

\end{abstract}


\end{frontmatter}

\section{Introduction}
\label{sec:intro}

Let $G$ be a finite group. Amiri et al. in \cite{amiri-communication} defined the following function:
$$\psi(G)= \sum_{g \in G} o(g).$$ 
where $o(g)$ denotes the order of the element $g$.  They were able to prove the following:

\begin{theorem}
\label{thm:amiri-cia} 
For any finite group $G$ of order $n$, $\psi(G) \leq \psi(\mathbb{Z}_n)$ and equality holds if and only if $G \simeq \mathbb{Z}_n,$ where $\mathbb{Z}_n$ denotes the cyclic group of order $n$.
\end{theorem}

That is, $\mathbb{Z}_n$ is the unique group of order $n$ with the largest value of $\psi(G)$ for groups
of that order. Later Amiri et al. in 
\cite{amiri-pureandApplied} and, independently,
Shen et al. in \cite{shen-et-al} investigated the groups with the second largest value of the sum of element orders.  
This function $\psi$ has been considered in various works (see \cite{amiri-communication-secondmax, amiri-algapplctn, chew-chin-lim, asad-joa, Her-joa, Her-cia, Her-jpaa,  tarnauceanu-israel}). While the goal of some of the papers was to find out the largest, second largest, least possible values of $\psi(G)$, the aim of some other papers was to prove new criteria for structural properties (like solvability, nilpotency, etc.) of finite groups.

Amiri and Amiri in \cite{amiri-amiri-cia} considered the following  generalization of the above function defined as:
$$\psi_k(G)= \sum_{g \in G} o(g)^k$$ 
for positive integers $k \geq 1$. Later, this function was also considered in \cite{suvra}.
We first note that
there exists positive integer $k>1$ and two groups $G_1$ and $G_2$ of same order such that $\psi(G_1)> \psi(G_2)$ but $\psi_k(G_1)< \psi_k(G_2).$ For example, one can take $G_1= D_{18}$ (the dihedral group consisting of $36$ elements) and $G_2= \mathbb{Z}_4 \times \Z_3 \times \Z_3$ and $k=6$ and a simple calculation shows that $\psi(G_1)=219<275=\psi(G_2)$ but 
$\psi_6(G_1)>\psi_6(G_2)$. Therefore, the comparison of the $\psi$-value of two groups can not directly tell about the comparison of $\psi_k$-value of two groups. Thus, for a fixed positive integer $n$, finding out the groups of order $n$ with the largest $\psi_k$ value is a valid question. Amiri et al. in \cite[Theorem 2.6]{amiri-amiri-cia} answered this question.

\begin{theorem}
\label{thm:Amiri_kth_gen}
Let $\mathbb{Z}_n$ be the cyclic group of order $n$ and $k$ be any positive integer. Then $\psi_k(G) < \psi_k(\mathbb{Z}_n)$  for all
non-cyclic groups $G$ of order $n$. 
\end{theorem}

In this paper,
 we start with the following result on the class of nilpotent groups. 

\begin{theorem}
\label{thm:Amiri_jaa_gen} 
Let $G$ be a nilpotent group of order $n$ and $k$ be any fixed positive integer. Then $\psi_k(G)$ is minimum among all possible nilpotent groups of order $n$ if and only if each Sylow subgroup of $G$ has prime
exponent. 
\end{theorem} 

We can see that Theorem \ref{thm:Amiri_kth_gen}
and Theorem \ref{thm:Amiri_jaa_gen} gives results on the extreme values of $\psi_k(G).$ We next concentrate on the class of finite abelian groups and there we prove the following recurrence relation with help of which it is easy to give an algorithm to find out the exact value of $\psi_k(G)$
for any finite abelian group $G$. 

\begin{theorem}
\label{thm:recurrence}
For positive integers $r_1 \leq r_2 \leq \dots \leq r_{t}$,
we have 
\vspace{1 mm} 
\begin{eqnarray}
& & \psi_k( \mathbb{Z}_{p^{r_1}} \times \mathbb{Z}_{p^{r_2}} \times  \dots \times 
\mathbb{Z}_{p^{r_t}})  \nonumber \\
& = & \psi_k( \underbrace{\mathbb{Z}_{p^{r_1}} \times \mathbb{Z}_{p^{r_1}} \times  \dots \times 
\mathbb{Z}_{p^{r_1}}}_\text {t \hspace{1 mm} \text{times} }  ) +  
  p^{r_1} \bigg ( \psi_k(  \mathbb{Z}_{p^{r_2}} \times  \dots \times 
\mathbb{Z}_{p^{r_t}}) - \psi_k( \underbrace{\mathbb{Z}_{p^{r_1}} \times \mathbb{Z}_{p^{r_1}} \times  \dots \times 
\mathbb{Z}_{p^{r_1}}}_\text {t-1 \hspace{1 mm} \text{times} } ) \bigg)\nonumber  \\
& & \label{eqn:recurrence}
\end{eqnarray}
where, 
$$\psi_k( \underbrace{\mathbb{Z}_{p^{r_1}} \times \mathbb{Z}_{p^{r_1}} \times  \dots \times 
\mathbb{Z}_{p^{r_1}}}_\text {s \hspace{1 mm} \text{times} }  )= \displaystyle \frac{p^{sr+s+rk+k} -p^{sr+kr+k}+p^k-1 }{p^{s+k}-1}  .$$
\end{theorem}

From Theorem \ref{thm:amiri-cia}, it is natural to ask whether there exists some constant $C<1$ such that $\psi(G) \leq C\psi(\mathbb{Z}_n)$ where $G$ is a non-cyclic group of order $n$.
Herzog et al. in \cite{Her-pureandApplied} proved the following remarkable result, giving an answer to this question. 

\begin{theorem}
\label{thm:Herzog_et_al_main}
Let $G$ be a non-cyclic group of order $n.$ Then $$\psi(G) \leq \frac{7}{11} \psi(\mathbb{Z}_n).$$ 
\end{theorem}



The main result of this paper is the following result which is ofcourse stronger than Theorem \ref{thm:Amiri_kth_gen} and also immediately gives Theorem \ref{thm:Herzog_et_al_main} if we set $k=1$ in \eqref{eqn:main_result}. 

\begin{theorem}
\label{thm:711generalization} 
Let $G$ be a non-cyclic group of order $n$ and $k$ be any fixed positive integer. Then,
\begin{equation}
\label{eqn:main_result}
\psi_k(G) \leq \frac{1+3.2^k}{1+2^k+2.4^k} \psi_k(\mathbb{Z}_n).
\end{equation}
\end{theorem}

Throughout this article, $\phi$ is the Euler's function. For a group $G$, the notation $\text{exp}(G)$  denotes the 
exponent of $G$ which is the smallest positive integer $x$ such that $g^x=e_G$ where $e_G$ denotes the identity element of $G$. Most of our notation is standard and we refer the reader to the books \cite{Isac-ams, Scott}.

\section{Preliminaries} 

In this Section, we recall some earlier known result
and also prove some preliminary lemmas which will be crucial in the forthcoming sections. Amiri et al. in \cite[Lemma 2.5]{amiri-amiri-cia} proved the following. 

\begin{lemma}
\label{lem:coprime}
For any fixed positive integer $k$ and any two finite groups $A$ and $B$, we have 
$\psi_k(A \times B) \leq \psi_k(A) \times \psi_k(B).$ Moreover, equality holds if and only if 
 $\gcd(|A|, |B|)=1$. 
\end{lemma}

We now consider the semidirect product of two groups and here we quote the following result from \cite[Lemma 2.2]{Her-pureandApplied}.

\begin{lemma}\label{semidirect product}
Let $G$ be a finite group satisfying $G=P \rtimes F$, where $P$ is a cyclic $p$-group for some prime $p$, $|F|>1$ and $(p,|F|)=1$. Then the following statements hold:
\begin{enumerate}
    \item Each element of $F$ acts on $P$ either trivially or fixed-point-freely.
    \item If $x\in F, \hspace{2 mm}o(x)=m$ and $u\in P$, then $m$ is the least positive integer satisfying $(ux)^m\in P$.
    \item If $u\in P$ and $x\in C_{F}(P)$, then $o(ux)=o(u)o(x)$.
    \item If $u\in P$ and $x\in F \setminus C_{F}(P)$, then $o(ux)=o(x)$.
    \item Let $Z=C_{F}(P)$. Then
    \begin{equation*}
    \psi(G)=\psi(P)\psi(Z)+|P|\psi(F \setminus  Z)<\psi(P)\psi(Z)+|P|\psi(F).
    \end{equation*}
\end{enumerate}
\end{lemma} 

In an identical way we can prove the following:

\begin{remark}
\label{rem:semidirect_product_analog} 
Let  $G$ be a finite group satisfying $G=P \rtimes F$, where $P$ is a cyclic $p$-group for some prime $p$, $|F|>1$ and $(p,|F|)=1$. Then, 
\begin{equation*}
    \psi_k(G)=\psi_k(P)\psi_k(Z)+|P|\psi_k(F \setminus  Z)<\psi_k(P)\psi_k(Z)+|P|\psi_k(F).
    \end{equation*}
\end{remark}

The next lemma is also easy to prove. Yet we provide the details for the sake of completeness.

\begin{lemma}
\label{lem:sumofkthpower} 
Let $G$ be a cyclic group of order $p^r$ for some prime $p$. Then,
$$\psi_k(G)= \displaystyle \frac{p^{kr+k+r+1}-p^{kr+k+r}+p^k-1}{p^{k+1}-1}= \displaystyle \frac {p^k.p^{r(k+1)} + (1+p+\dots+p^{k-1}) } {1+p+\dots+p^k}  . $$
Let $G =  \underbrace{\mathbb{Z}_{p^{r}} \times \mathbb{Z}_{p^{r}} \times  \dots \times 
\mathbb{Z}_{p^{r}}}_\text {s \hspace{1 mm} \text{times} } . $ We then have 
$$ \psi_k(G)= \displaystyle \frac{p^{sr+s+rk+k} -p^{sr+kr+k}+p^k-1 }{p^{s+k}-1}  .$$
\end{lemma}

\begin{proof}
We first consider the case when $G$ is cyclic and of order $p^r.$ The main thing is to note that the number of elements of order $p^i$ in $G$ is $\phi(p^i).$ Thus, we have 
$$\psi_k(G)= 1+p^k\phi(p)+ \dots+p^{rk}\phi(p^r).$$
The remaining follows from computation and hence is omitted. 

For the second part, we observe that the number of elements of order $p^i$ is $p^{is}-p^{(i-1)s}.$ Here we do the calculation for the sake of completeness. We have 
\begin{eqnarray*}
\psi_k(G) &  = & 1+ (p^s-1) p^k + (p^{2s}-p^s) p^{2k} + \cdots + (p^{rs}-p^{(r-1)s})p^{rk} \\
& = & 1+ p^{s+k}+p^{2(s+k)}+\dots + p^{r(s+k)} - [ p^k+p^{s+2k} + \dots+ p^{(r-1)s+rk} ] \\
& = & \displaystyle \frac{p^{(s+k)(r+1)}-1}{p^{s+k}-1} - p^k\frac{p^{(s+k)r}-1}{p^{s+k}-1} \\
& = & \displaystyle \frac{p^{sr+s+rk+k} -p^{sr+kr+k}+p^k-1 }{p^{s+k}-1}
\end{eqnarray*}
This completes the proof. 
\end{proof}

One can note that using Lemma \ref{lem:coprime} and Lemma \ref{lem:sumofkthpower}, one can explicitly compute $\psi_k(G)$ for any cyclic group $G$.


We next provide a lower bound of the $\psi_k$-value of any cyclic group in terms of the highest and the lowest prime divisors of the order of the group. 

\begin{lemma}
\label{lem:lowerbound}
Let $q=p_1 < p_2 < \cdots < p_t = p$ be the prime divisors of $n$ and the corresponding Sylow subgroups of
$\mathbb{Z}_n$ are $P_1, P_2, \dots, P_t.$ Then 
$$\psi_k(\mathbb{Z}_n) > \displaystyle \frac{q^k}{1+p+\dots+p^k} n^{k+1} > \displaystyle \bigg(\frac{q}{p+1}\bigg)^kn^{k+1}.$$ 
\end{lemma}

\begin{proof} 
As $\mathbb{Z}_n=P_1 \times P_2 \times \cdots P_t$, by applying Lemma \ref{lem:coprime}, we have 
\begin{eqnarray*}
\psi_k(\mathbb{Z}_n) & = & \prod_{i=1}^t \psi_k (\mathbb{Z}_{p_i^{r_i}} ) \\
& \geq &
\displaystyle \prod _{i=1} ^ {t} \frac{p_i^k} {1+p_i+\dots+p_i^k} |P_i|^{k+1}.  \\
& \geq & \displaystyle \frac{q^k} {1+p+\dots+p^k} \prod _{i=1} ^ {t} |P_i| ^ {k+1} =  \displaystyle \frac{q^k} {1+p+\dots+p^k} n^{k+1}. 
\end{eqnarray*}
The second line follows by using Lemma \ref{lem:sumofkthpower}. The third line uses the fact that $p_{i+1}\geq p_i+1$ and hence $p_{i+1}^k > p_i^k+p_i^{k-1}+\dots+1.$ Thus, the first inequality of this Lemma is proved. The second is immediate. 
\end{proof}

Before going to the next lemma, we need the following result, which is implicit in the proof of Lemma A of \cite{amiri-communication}.  

\begin{lemma}
\label{lem:implicit_communication}
Let $P$ be a Sylow p-subgroup of $G$, $P \trianglelefteq G$ and $P$ is cyclic in $G$. Let $m$ be the order of the element $Px.$ Every element of $Px$ clearly has the form $ux$ for some element $u \in P.$ We have
$ o(ux) \leq m(o(u))$ with equality if and only if $x$ centralizes $u.$ 
\end{lemma}

Using Lemma \ref{lem:implicit_communication}, we now prove the following: 

\begin{lemma}
\label{lem:analogue_cia} 
Let $P$ be a Sylow $p$-subgroup of $G$, $P \trianglelefteq G$ and $P$ is cyclic in $G$. Then, 
$$\psi_k(G) \leq \psi_k(P) \psi_k(G/P).$$ 
Moreover, equality happens if and only if $P \subseteq Z(G).$
\end{lemma}

\begin{proof}
Let $x \in G$ and we consider $Px$ as an element of $G/P.$ We at first claim the following: 

\vspace{1 mm} 

{\bf Claim:}  We have $ \psi_k(Px) \leq o(Px)^k \psi_k (P)$ and equality holds if and only if $x$ centralizes $P$. 

\vspace{1 mm} 
{\bf Proof of claim:} We have 
\begin{eqnarray*}
\psi_k(Px) & = & \sum _{u \in P} o(ux)^k \leq \sum _{u \in P} o(Px)^k o(u) ^k \\
& = & o(Px)^k \sum_{u \in P} o(u)^k = o(Px)^k \psi_k(P).
\end{eqnarray*} 
Equality holds in the above equation if and only if $x$ centralizes $u$ for every element $u$ in $P$, that is, when $x$ centralizes $P$.  

With this claim, we now have 
\begin{eqnarray*}
\psi_k(G) & = & \sum _{Px \in G/P} \psi_k(Px)  \\
& \leq & \sum _{Px \in G/P}  o(Px) ^k \psi_k(P) =  \psi_k(P)\sum_{Px \in G/P}  o(Px) ^k = \psi_k(P) \psi_k(G/P). 
\end{eqnarray*} 
Clearly, equality holds if and only if every element $x$ in $G$ centralizes $P$, or equivalently $P \subseteq Z(G).$
\end{proof}

\section{Nilpotent and abelian groups} 

We note that if $G$ is a $p$-group of order, say $p^r$,  then we of course have $\psi_k (G) \geq (p^r-1) p^k+1$ and equality holds if and only if every non-identity element of $G$ has order $p$. With this observation, we now move to the proof of Theorem \ref{thm:Amiri_jaa_gen}.

\begin{proof}[Proof of Theorem \ref{thm:Amiri_jaa_gen}]
Let $G$ be a nilpotent group of order $n>1$ and moreover, let
$$
G= P_1 \times P_2 \times \dots P_t,$$
where $P_i$ is the Sylow $p_i$-subgroup of $G$ and for all $1 \leq i \leq t$ we have $|P_i|=p_i^{r_i} $ with $r_i \in \mathbb{N}$  and $\gcd(p_i, p_j)=1$ for $i \neq j.$ By Lemma \ref{lem:coprime} we have $\psi_k(G)= \psi_k(P_1) \times \psi_k(P_2) \dots \psi_k(P_t)$ and by the above observation $\psi_k(P_i)$ is minimum if and only if every non-identity element of $P_i$ is of prime order. Thus, $\psi_k(G)$ will be minimum if and only if each Sylow subgroup of $G$ has prime order. This completes the proof. 
\end{proof}

We next consider the family of abelian groups and and as any abelian group can be written as a direct product of abelian $p$-groups, determining the $\psi_k$ value explicitly for any abelian $p$-group is sufficient in order to determine the $\psi_k$ value of any abelian group. In this context, we have the following. 

\begin{proof}
[Proof of Theorem \ref{thm:recurrence}]
Define $$A= \{ (a_1, a_2, \dots, a_t) \in \mathbb{Z}_{p^{r_1}} \times \mathbb{Z}_{p^{r_2}} \times \dots \times  \mathbb{Z}_{p^{r_t}} : a_i \equiv 0 \hspace{1 mm} \mod p^{r_i-r_1} \hspace{1 mm} \mbox{ for each } i  \}$$
Clearly, $A$ is a subgroup of $\mathbb{Z}_{p^{r_1}} \times \mathbb{Z}_{p^{r_2}} \times \dots \times    \mathbb{Z}_{p^{r_t}}.$ Define $$f:  \underbrace{\mathbb{Z}_{p^{r_1}} \times \mathbb{Z}_{p^{r_1}} \times  \dots \times 
\mathbb{Z}_{p^{r_1}}}_\text {t \hspace{1 mm} \text{times} } \mapsto A $$ by
$$f(x_1, x_2, \dots, x_t) = (x_1, x_2p^{r_2-r_1}, x_3 p^{r_3-r_1}, \dots, x_tp^{r_t-r_1}).$$ 
It is easy to check that this map $f$ is reversible and it is also clearly a group homomorphism. Thus, $\mathbb{Z}_{p^{r_1}} \times \mathbb{Z}_{p^{r_1}} \times  \dots \times 
\mathbb{Z}_{p^{r_1}} \simeq A. $ 
Thus, 
\begin{equation*}
\psi_k (\mathbb{Z}_{p^{r_1}} \times \mathbb{Z}_{p^{r_2}} \times \dots \times  \mathbb{Z}_{p^{r_t}} ) = \psi_k(\underbrace{\mathbb{Z}_{p^{r_1}} \times \mathbb{Z}_{p^{r_1}} \times  \dots \times 
\mathbb{Z}_{p^{r_1}}}_\text{t \text{times}})+\psi_k(\mathbb{Z}_{p^{r_1}} \times \mathbb{Z}_{p^{r_2}} \times \dots \times  \mathbb{Z}_{p^{r_t}}-A)
\end{equation*} 
Now, we have 
\begin{eqnarray*}
   & & \psi_k(\mathbb{Z}_{p^{r_1}} \times \mathbb{Z}_{p^{r_2}} \times \dots \times  \mathbb{Z}_{p^{r_t}}-A) \\ & = & \sum _{(a_1, a_2, \dots, a_t) \in \mathbb{Z}_{p^{r_1}} \times \mathbb{Z}_{p^{r_2}} \times \dots \times  \mathbb{Z}_{p^{r_t}}}  o (a_1, a_2, \dots, a_t)^k - 
   \sum _{(a_1, a_2, \dots, a_t) \in A } o (a_1, a_2, \dots, a_t)^k \\ 
    & = & p^{r_1}  \bigg [ \sum _{( a_2, \dots, a_t) \in  \mathbb{Z}_{p^{r_2}} \times \dots \times  \mathbb{Z}_{p^{r_t}}}  o ( a_2, \dots, a_t)^k - 
   \sum _{( a_2, \dots, a_t) \in A } o ( a_2, \dots, a_t)^k \bigg] \\
   & = & p^{r_1} \bigg ( \psi_k(  \mathbb{Z}_{p^{r_2}} \times  \dots \times 
\mathbb{Z}_{p^{r_t}}) - \psi_k( \underbrace{\mathbb{Z}_{p^{r_1}} \times \mathbb{Z}_{p^{r_1}} \times  \dots \times 
\mathbb{Z}_{p^{r_1}}}_\text {t-1 \hspace{1 mm} \text{times} } ) \bigg) 
\end{eqnarray*}
This completes the proof of Equation \eqref{eqn:recurrence}. The proof of this Theorem is now complete by using Lemma \ref{lem:sumofkthpower}. 
\end{proof}

It is worth mentioning here that Saha \cite[Proposition 9]{suvra}
proved the following recursive formula.

\begin{theorem}
\label{thm:Saha_finite_abelian_pgroup}
Let $G = \mathbb{Z}_{p^r} \times H$ where $r \geq 1$ and $H$ is a $p$-group with $\text{exp}(H) \geq p^r.$ Let $N_j$ be the number of elements in $H$ that have order $p^j.$ Then, 
\begin{eqnarray*}
\psi_k(G)=\begin{cases} 
p^r\psi_k(H) + \displaystyle \sum_{i=2}^r \bigg[ (p^i-p^{i-1}) \big[(p^{ki}-1)+\sum_{j=1}^{i-1}(p^{ki}-p^{kj} )N_j \big] \bigg] &  \\ + (p-1)(p^k-1), 
& \text{ if $r>1$} \\  & \\

p\psi_k(H)+(p-1)(p^k-1), & \text{ if $r=1$} 
 \end{cases}
\end{eqnarray*}

\end{theorem}
It is clear that for any abelian $p$-group with $r=1$, Theorem \ref{thm:Saha_finite_abelian_pgroup} is better than Theorem \ref{thm:recurrence}. But for $r>1$, the expression of $\psi_k(G)$ in Theorem 
\ref{thm:recurrence} contains nontrivial summations  where as the expression of $\psi_k(G)$ in Theorem \ref{thm:recurrence} is {\it neat} in the following sense that it involves only $3$ terms. 

We can now compute $\psi_k(G)$ for any finite abelian group $G$ using Lemma \ref{lem:coprime} and Theorem \ref{thm:recurrence}. 


\section{An important Lemma and some inequalities} 

The main focus of the remaining part of the paper is proving Theorem \ref{thm:711generalization} and in that direction, we require the following important lemma.  

\begin{lemma}
\label{lem:k-upperbound}
Let $G$ be a non-cyclic finite group of order $n$ and let $q$ be the smallest prime divisor of $n$. Then we have
\begin{center}
    $\psi_{k}(G)< \displaystyle \frac{1}{(q-1)^k}\psi_{k}(\mathbb{Z}_{n}).$
\end{center}
\end{lemma}

\begin{proof}
 It is easy to show that $\phi(n)\geq \displaystyle \frac{q-1}{p}n$, where $p$ and $q$ are respectively the largest prime and the smallest prime dividing $n$. So we have
 \begin{center}
     $\phi(n)\geq \displaystyle \frac{q-1}{p}n\geq \bigg(\frac{q-1}{p}\bigg)^kn$
 \end{center} 
 Now we need to prove that if $\psi_{k}(G)\geq\displaystyle \frac{1}{(q-1)^k}\psi_{k}(\mathbb{Z}_n)$ then $G$ is a cyclic group of order $n$.
 \medskip
 
As there are $\phi(n)$ many elements of order $n$, we clearly have, $\psi_{k}(\mathbb{Z}_{n})>n^k\phi(n)$. Hence, by our assumption,
 \begin{equation*}
     \psi_{k}(G) \geq \displaystyle \frac{1}{(q-1)^k}\psi_{k}(\mathbb{Z}_n)>\frac{n^k}{(q-1)^k}\phi(n)\geq \bigg(\frac{q-1}{p}\bigg)^k\frac{n^{k+1}}{(q-1)^k}=\frac{n^{k+1}}{p^k}.
 \end{equation*}
 
 This implies that there exists an element $x\in G$ such that $o(x)>\displaystyle\frac{n}{p}$. Thus $[G:\langle x\rangle]<p$ and $\langle x\rangle$ contains a Sylow $p$-subgroup $P$ of $G$. Since $\langle x\rangle\leq N_{G}(P)$, it follows that $P$ is a cyclic normal subgroup of $G$ and by using Lemma \ref{lem:analogue_cia} we obtain,
 \begin{center}
 $\psi_{k}(P)\psi_{k}(G/P)\geq \psi_{k}(G)\geq \displaystyle\frac{1}{(q-1)^k}\psi_{k}(\mathbb{Z}_{p^r})\psi_{k}(\mathbb{Z}_{n/p^r})$,
 \end{center}
 where $p^r=|P|$. Since $P\simeq \mathbb{Z}_{p^r}$, by cancellation we get 
 \begin{equation*}
\displaystyle \psi_{k}(G/P)\geq \frac{1}{(q-1)^k}\psi_{k}(\mathbb{Z}_{n/p^r}).
 \end{equation*}
 \medskip
 
 If $n=p^r$, $p$ a is prime, then the existence of $x\in G$ satisfying $o(x)>n/p$ implies that $o(x)=n$ and $G$ is cyclic, as required. So we may assume that $n$ is divisible by exactly $t$ different primes with $t>1$. Applying induction with respect to $t$, we may assume that the theorem holds for groups of order which has less than $t$ distinct prime divisors. Since $|G/P|$ has $t-1$ distinct prime divisors and $G/P$ satisfies our assumptions, it follows that $|G/P|$ is cyclic and $G=P\rtimes F$, with $F\simeq G/P$ and $F\neq 1$. Notice that $n=|P||F|$, where $ P$ and $F$ are both cyclic and $\gcd(|P|,|F|)=1$. So $\psi_{k}(\mathbb{Z}_{n})=\psi_{k}(P)\psi_{k}(F)$.
 \medskip
 
 If $C_{F}(P)=F$, then $G=P\times F$ and $G$ is cyclic, as required.
 \medskip 
 
 So it suffices to prove that if $C_{F}(P)=Z<F$, then $\psi_{k}(G)<\displaystyle\frac{1}{(q-1)^k}\psi_{k}(\mathbb{Z}_{n})$, which is contrary to  assumptions. Using Lemma \ref{semidirect product} we have
 \begin{center}
     $\psi_{k}(G)=\psi_{k}(P)\psi_{k}(Z)+|P|\psi_{k}(F\setminus Z) < \psi_k(P)\psi_k(Z) +|P|\psi_{k}(F)$.
 \end{center}
 Hence 
 \begin{equation*}
     \psi_{k}(G)<\psi_{k}(P)\psi_{k}(F)\bigg(\frac{\psi_{k}(Z)}{\psi_{k}(F)}+\frac{|P|}{\psi_{k}(P)}\bigg)=\psi_{k}(\mathbb{Z}_{n})\bigg(\frac{\psi_{k}(Z)}{\psi_{k}(F)}+\frac{|P|}{\psi_{k}(P)}\bigg).
 \end{equation*}
 Now as $P$ is a cyclic $p$-group, we have,
 \begin{eqnarray}
 \frac{|P|}{\psi_{k}(P)}&=&\frac{|P|(p^{k+1}-1)}{|P|^{k+1}p^{k}(p-1)+(p^k-1)}\nonumber\\
 &=&\frac{|P|(p^k+p^{k-1}+\cdots+1)}{|P|^{k+1}p^k+(p^{k-1}+\cdots+1)}\nonumber\\
 &<&\frac{|P|(p^k+p^{k-1}+\cdots+1)}{|P|^{k+1}p^{k}}\nonumber\\
 &<& \frac{(p^k+\cdots+1)}{p^{2k}-p^{k-1}}\nonumber\\
 &=& \frac{(p-1)(p^k+\cdots+1)}{(p-1)p^{k-1}(p^{k+1}-1)}=\frac{1}{(p-1)(p^{k-1})}<\frac{1}{q}\cdot\frac{1}{q^{k-1}}=\frac{1}{q^{k}}\nonumber
\end{eqnarray}

Next note that $Z$ is a proper subgroup of the cyclic group $F$ and $\psi_{k}(F)$ is a product of $\psi_{k}(S)$, with $S$ running over all the Sylow subgroups of $F$. Moreover, $\psi_{k}(Z)$ is a similar product,  and at least one Sylow subgroup of $Z$, say Sylow $p_1$-subgroup $R_{Z}$, is properly contained in the Sylow $p_1$-subgroup $R_{F}$ of $F$  and let $R_F$ is of order $p_1^s$ for some prime $p_1$. Here $F$ is a cyclic group of order, say, $p_1^s|Q|$, where $Q$ is the maximal cyclic subgroup of $F$ with $\gcd(p_1,|Q|)=1$. Since,  the Sylow $p_1$-subgroup of $Z$, $R_{Z}$, is properly contained  in the cyclic group of $R_{F}$, the order of $R_{Z}$ is less than equal to $p_1^{s-1}$. Moreover, $Z$ is a cyclic subgroup properly contained in $F$.  Hence order of $Z$ is less equal to
$|R_{Z}||Q|$, which is equal to  $p_1^{s-1}|Q|$. So, applying Lemma \ref{lem:coprime}, we obtain the following:
\begin{eqnarray}
    &&\psi_{k}(Z)\leq \psi_{k}(R_{Z})\psi_{k}(Q)
    \leq\psi_{k}(R_{F})\psi_{k}(Q)= \psi_{k}(F)\nonumber\\
    &\Rightarrow& [\psi_{k}(Z)][\psi_{k}(R_{F})\psi_{k}(Q)]\leq [\psi_{k}(R_{Z})\psi_{k}(Q)][\psi_{k}(F)]\nonumber\\
    &\Rightarrow& \psi_{k}(Z)\psi_{k}(R_{F})\leq \psi_{k}(F)\psi_{k}(R_{Z})\nonumber\\
   &\Rightarrow& \frac{\psi_{k}(Z)}{\psi_{k}(F)}\leq \frac{\psi_{k}(R_{Z})}{\psi_{k}(R_{F})}\leq \frac{(p_1^{s-1})^{k+1}p_1^k(p_1-1)+(p_1^k-1)}{(p_1^s)^{k+1}p_1^k(p_1-1)+(p_1^k-1)},\label{similar}
\end{eqnarray} 
where, the last inequality is obtained by Lemma \ref{lem:sumofkthpower}. if $s=1$, it is easy to see that 
$$\frac{p_1^k(p_1-1)+(p_1^k-1)}{(p_1)^{k+1}p_1^k(p_1-1)+(p_1^k-1)} \leq \frac{1}{p_1^{k+1}-p_1^k} \leq \frac{1}{q^{k+1}-q^k}.$$
If $s >1$, we have 
\begin{eqnarray}
\frac{(p_1^{s-1})^{k+1}p_1^k(p_1-1)+(p_1^k-1)}{(p_1^s)^{k+1}p_1^k(p_1-1)+(p_1^k-1)} & \leq &  \frac{1}{p_1^{k+1}-p_1^k} \frac{\big[(p_1^{s-1})^{k+1}p_1^k(p_1-1)+(p_1^k-1)\big](p_1^{k+1}-p_1^k)}{(p_1^s)^{k+1}p_1^k(p_1-1)+(p_1^k-1)} \nonumber \\
& \leq  & \frac{ (p_1^s)^{k+1}p_1^k(p_1-1)-p_1^{(s-1)(k+1)+2k}(p_1-1)+p_1^{2k+1}}{(p_1^{k+1}-p_1^k)(p_1^s)^{k+1}p_1^k(p_1-1))} \nonumber \\
& \leq & \frac{1}{p_1^{k+1}-p_1^k} \leq \frac{1}{q^{k+1}-q^k}.
\label{eqn:imp_lemma} 
\end{eqnarray}
The last line is obtained by using the fact that $p_1 \geq q.$
Finally, we obtain that,
\begin{eqnarray}
    \psi_{k}(G)<\psi_{k}(\mathbb{Z}_{n})\bigg(\frac{\psi_{k}(Z)}{\psi_{k}(F)}+\frac{|P|}{\psi_{k}(P)}\bigg)&<&\psi_{k}(\mathbb{Z}_{n})\bigg(\frac{1}{q^k}+\frac{1}{q^{k+1}-q^k}\bigg)\nonumber\\
    &=&\psi_{k}(\mathbb{Z}_{n})\frac{1}{(q-1)q^{k-1}}<\psi_{k}(\mathbb{Z}_{n})\frac{1}{(q-1)^k},\nonumber
\end{eqnarray}
which is a contradiction. This completes the proof.
\end{proof}

The above lemma immediately gives the following corollary.

\begin{corollary}
\label{cor:odd_order_group}
Let $G$ be a non-cyclic finite group of odd order $n$. Then, we have
$$\psi_k(G) < \displaystyle  \frac{1}{2^k} \psi_k(\mathbb{Z}_n).$$ 
\end{corollary}

We now prove the following three inequalities which are important ingredients for the proof of Theorem \ref{thm:711generalization}.

\begin{lemma} 
\label{lem:ineqlemma_1}
Let $k$ be any positive integer and $n=2^a3^b$ for some $a,b \in \mathbb{N}.$
Then, the following inequality holds:   \begin{equation*}\psi_{k}(\mathbb{Z}_{n/2}) + \displaystyle \bigg(\frac{n}{2} \bigg) \bigg(\frac{n}{3} \bigg)^k \leq \frac{1+3.2^k}{1+2.4^k+2^k}\psi_{k}(\mathbb{Z}_n).
\end{equation*} 
 \end{lemma}

\begin{proof} 
Setting $n=2^a3^b$, we see that is equivalent to showing

\begin{equation}
\label{eqn:claim1_eq1} 
\psi_{k}(\mathbb{Z}_{2^{a-1}})\psi_{k}(\mathbb{Z}_{3^b})+ 2^{a+ak-1}3^{b+bk-k} \leq \frac{1+3.2^k}{1+2.4^k+2^k}\psi_{k}(\mathbb{Z}_{2^a})\psi_{k}(\mathbb{Z}_{3^b}). 
\end{equation}

This boils down to showing
\begin{equation}
\label{eqn:claim1_eq2} 
2^{a+ak-1}3^{b+bk-k} \leq 
\frac{\psi_{k}(\mathbb{Z}_{3^b})}{(1+2.4^k+2^k)}\bigg[(1+3.2^k)\psi_{k}(\mathbb{Z}_{2^a})-(1+2.4^k+2^k)\psi_{k}(\mathbb{Z}_{2^{a-1}})\bigg]. 
\end{equation}

We see that there are $2^{a-1}$ elements of order $2^a$ in the group $\mathbb{Z}_{2^{a}} $ and
the remaining $2^{a-1}$ elements form the group $\mathbb{Z}_{2^{a-1}}. $ Therefore, we have 
$\psi_{k}(\mathbb{Z}_{2^a})=\psi_{k}(\mathbb{Z}_{2^{a-1}})+ 2^{ak+a-1}.$
We use this to obtain that showing \eqref{eqn:claim1_eq2} is equivalent to show 
\begin{equation}
\label{eqn:claim1_eq3} 2^{a+ak-1}3^{b+bk-k}(1+2.4^k+2^k)+\psi_{k}(\mathbb{Z}_{2^{a-1}})\psi_{k}(\mathbb{Z}_{3^b})(2.4^k-2.2^k) \leq  \psi_{k}(\mathbb{Z}_{3^b})(1+3.2^k)2^{ak+a-1}. 
\end{equation}  

For this, we will show the following two inequalities 
\begin{equation}
\label{eqn:claim1_eq4} 
2^{a+ak-1}3^{b+bk-k}(1+2.4^k+2^k)\leq \psi_{k}(\mathbb{Z}_{3^b})(1+2^k)2^{ak+a-1}
\end{equation}
and 
\begin{equation}
\label{eqn:claim1_eq5}
\psi_{k}(\mathbb{Z}_{2^{a-1}})\psi_{k}(\mathbb{Z}_{3^b})2^{k+1}(2^k-1)\leq\psi_{k}(\mathbb{Z}_{3^b})(2.2^k)2^{ak+a-1}.
\end{equation}

For all positive integers $k$, we of course have (i) $2^{k+1}(1+2^k) \geq (1+2.4^k+2^k)$ and for $ k \geq 3,$ (ii) we have $2^k \leq 3^{k-1}.$ We also note that there are $2.3^{b-1} $ elements of order $3^b$ in the group $\mathbb{Z}_{3^b}$ and hence (iii) $\psi_k(\mathbb{Z}_{3^b}) \geq 2.3^{b-1}.3^{bk}. $ Multiplying (i), (ii) and (iii), we have that for $k \geq 3$, $\eqref{eqn:claim1_eq4}$ is true. 

We also see that, $$\frac{2^{a+ak-1}}{2^k-1}\geq \frac{2^{a+ak-1}}{2^k}=2^{(a-1)(k+1)}$$ and  $$\psi_{k}(\mathbb{Z}_{2^{a-1}})\leq 2^{(a-1)(k+1)}.$$ Multiplying above two equations, we get that \eqref{eqn:claim1_eq5} is proved. 

Thus, we are able to prove that claim 1 holds for $k\geq 3$.

Inserting $k=2$ in equation \eqref{eqn:claim1_eq3}, we need to show that 

\begin{equation}
\label{eqn:claim1_eq6} 
37.2^{3a-1} 3^{3b-2}+ 24 \psi_2(\mathbb{Z}_{2^{a-1}}) \psi_2(\mathbb{Z}_{3^b}) \leq 13.2^{3a-1}  \psi_2(\mathbb{Z}_{3^b}).
\end{equation}

We have
\begin{eqnarray*}
&&\psi_{2}(\mathbb{Z}_{3^b})\bigg[13.2^{3a-1}-24.\psi_{2}(\mathbb{Z}_{2^{a-1}})\bigg]\nonumber\\
&=&\bigg[\frac{18.3^{3b}+8}{26}\bigg].\bigg[13.2^{3a-1}-24.\frac{(2^{3a-1}+3)}{7}\bigg]\nonumber\\
&=& \bigg[\frac{18.3^{3b}+8}{26}\bigg].\bigg[\frac{91.2^{3a-1}-24.2^{3a-1}-72}{7}\bigg]\nonumber\\
&=&\bigg[\frac{18.3^{3b}+8}{26}\bigg].\bigg[\frac{67.2^{3a-1}-72}{7}\bigg]\nonumber \\
& \geq & \displaystyle \frac{ 18.9 }{26} .3^{3b-2}.7.2^{3a-1} =42.2^{3a-1} 3^{3b-2}.
\end{eqnarray*}
The second line follows from Lemma \ref{lem:sumofkthpower} and the last line uses the fact that for $a\geq 1$, $3a-1\geq2$ and $18.2^{3a-1}\geq 18.4=72.$ Hence, for $k= 2,$ we are able to prove \eqref{eqn:claim1_eq6}. For $k=1$, this can either be proved
in a similar manner or this follows from the proof of \cite[Theorem 1]{Her-pureandApplied}. 

The proof is complete.
\end{proof} 

 \begin{lemma}
 \label{lem:ineqlemma2} 
 Let $k$ be any positive integer and $n=2^a3^b$ for some $a,b \in \mathbb{N}.$
 Then, we have 
  \begin{equation*} \psi_k(\mathbb{Z}_{n/3})+ \displaystyle  \frac{2n}{3}\bigg(\frac{n}{3} \bigg)^k \leq \frac{1+3.2^k}{1+2.4^k+2^k}\psi_{k}(\mathbb{Z}_n). 
  \end{equation*} 
  \end{lemma} 
 
 \begin{proof} The proof goes in a similar way to that of Lemma \ref{lem:ineqlemma_1}. But we provide the details for the sake of completeness. We set $n=2^a3^b$ and, it is now equivalent to showing

\begin{equation}
\label{eqn:claim2_eq1} 
\psi_{k}(\mathbb{Z}_{2^{a}})\psi_{k}(\mathbb{Z}_{3^{b-1}})+ 2^{a+ak+1}3^{b+bk-k-1} \leq \frac{1+3.2^k}{1+2.4^k+2^k}\psi_{k}(\mathbb{Z}_{2^a})\psi_{k}(\mathbb{Z}_{3^b}). 
\end{equation}

This boils down to showing
\begin{equation}
\label{eqn:claim2_eq2} 
2^{a+ak+1}3^{b+bk-k-1} \leq 
\frac{\psi_{k}(\mathbb{Z}_{2^a})}{(1+2.4^k+2^k)}\bigg[(1+3.2^k)\psi_{k}(\mathbb{Z}_{3^b})-(1+2.4^k+2^k)\psi_{k}(\mathbb{Z}_{3^{b-1}})\bigg]. 
\end{equation}

Here we see that there are $2.3^{b-1}$ elements of order $3^b$ in the group $\mathbb{Z}_{3^{b}} $ and
the remaining $3^{b-1}$ elements form the group $\mathbb{Z}_{3^{b-1}}. $ Therefore, we have 
$\psi_{k}(\mathbb{Z}_{3^b})=\psi_{k}(\mathbb{Z}_{3^{b-1}})+ 2.3^{bk+b-1}.$
We use this to obtain that showing \eqref{eqn:claim2_eq2} is equivalent to show 
\begin{equation}
\label{eqn:claim2_eq3} 2^{a+ak+1}3^{b+bk-k-1}(1+2.4^k+2^k)+\psi_{k}(\mathbb{Z}_{2^{a}})\psi_{k}(\mathbb{Z}_{3^{b-1}})(2.4^k-2.2^k) \leq  \psi_{k}(\mathbb{Z}_{2^a})(1+3.2^k)2.3^{bk+b-1}. 
\end{equation}  

For this, we will show the following two inequalities 
\begin{equation}
\label{eqn:claim2_eq4} 
2^{a+ak+1}3^{b+bk-k-1}(1+2.4^k+2^k)\leq \psi_{k}(\mathbb{Z}_{2^a})(1+2^k)2.3^{bk+b-1} 
\end{equation}
and 
\begin{equation}
\label{eqn:claim2_eq5}
\psi_{k}(\mathbb{Z}_{2^{a}})\psi_{k}(\mathbb{Z}_{3^{b-1}})2^{k+1}(2^k-1)\leq\psi_{k}(\mathbb{Z}_{2^a})(2.2^k)2.3^{bk+b-1}.
\end{equation}

For all positive integers $k$, we of course have (i) $2^{k+1}(1+2^k) \geq (1+2.4^k+2^k)$ and for $ k \geq 4,$ (ii) we have $2^{k+2}  \leq 3^{k}.$ We also note that there are $2^{a-1} $ elements of order $2^a$ in the group $\mathbb{Z}_{2^a}$ and hence (iii) $\psi_k(\mathbb{Z}_{2^a}) \geq 2^{a-1}. $ Multiplying (i), (ii) and (iii), we have that for $k \geq 4$, $\eqref{eqn:claim2_eq4}$ is true. 

Moreover, we ofcourse have $\psi_k(\mathbb{Z}_{3^{b-1}}) \leq 3^{b-1}.3^{(b-1)k} $ and this directly proves 
 \eqref{eqn:claim2_eq5}.

Thus, we are able to prove that claim 2 holds for $k\geq 4$.
We now prove for $k=3$ and $k=2$ separately. 

Inserting $k=3$ in equation \eqref{eqn:claim2_eq3}, we need to show that 

\begin{equation}
\label{eqn:claim2_eq6} 
137.2^{4a+1} 3^{4b-4}+ 120 \psi_3(\mathbb{Z}_{2^{a}}) \psi_3(\mathbb{Z}_{3^{b-1}}) \leq 25.2.3^{4b-1}  \psi_3(\mathbb{Z}_{2^a}) .
\end{equation}

We now have
\begin{eqnarray*}
&&\psi_{3}(\mathbb{Z}_{2^a})\bigg[50.3^{4b-1}-120.\psi_{3}(\mathbb{Z}_{3^{b-1}})\bigg]\nonumber\\
& \geq & 2^{4a+1} \bigg[50.3^{4b-1}-120.\frac{2.3^{4b-1}+26}{80}\bigg]\nonumber\\ 
&=& 2^{4a+1} \bigg[\frac{4000.3^{4b-1}-240.3^{4b-1}-3120}{80} \bigg] \nonumber\\
& \geq & 2^{4a+1}3^{4b-4} .27.10  \\ 
& > & 137. 2^{4a+1}.3^{4b-4}.
\end{eqnarray*}
The second line follows from Lemma \ref{lem:sumofkthpower} and the fourth line uses the fact that for $b\geq 1$, $4b-1\geq3$ and $150.3^{4b-1} \geq 150.27> 3120$. Hence, for $k= 3$ we are done.

Inserting $k=2$ in equation \eqref{eqn:claim2_eq3}, we need to show that 

\begin{equation}
\label{eqn:claim2_eq7} 
37.2^{3a+1} 3^{3b-3}+ 24 \psi_2(\mathbb{Z}_{2^{a}}) \psi_2(\mathbb{Z}_{3^{b-1}}) \leq 13.2.3^{3b-1}  \psi_2(\mathbb{Z}_{2^a}). 
\end{equation}

We now have
\begin{eqnarray*}
&&\psi_{2}(\mathbb{Z}_{2^a})\bigg[26.3^{3b-1}-24.\psi_{2}(\mathbb{Z}_{3^{b-1}})\bigg]\nonumber\\
& \geq & 2^{3a-1} \bigg[26.3^{3b-1}-24.\frac{18.3^{3b-3}+8}{26}\bigg]\nonumber\\ 
&=& 2^{3a-1} \bigg[\frac{676.3^{3b-1}- 48.3^{3b-1}-192}{26} \bigg] \nonumber\\
& \geq & 2^{3a-1} 24.9.3^{3b-3} \\ 
& > & 37. 2^{3a+1}3^{3b-3}.
\end{eqnarray*}
The second line follows from Lemma \ref{lem:sumofkthpower} and the last line uses the fact that for $b\geq 1$, $3b-1\geq2$ and $22.3^{3b-1}\geq 192.$. Hence, for $k= 2$ we are done. For $k=1$, this can be proved
in a similar manner or this follows from the proof of \cite[Theorem 1]{Her-pureandApplied}. 

This completes the proof.
\end{proof} 

\begin{lemma}
\label{lem:ineq3lemma} 
For primes $p>3$ and $q \geq 2$, we have $$\bigg( \frac{1+2.4^k+2^k}{1+3.2^k} \bigg)^{1/k}\frac{(1+p+\dots+p^k)^{1/k}}{q}  < p.$$
\end{lemma}

\begin{proof}
We need to show 
\begin{equation*}
    \label{eqn:claim3_eqn1} 
  \bigg( \frac{1+2.4^k+2^k}{1+3.2^k} \bigg)^{1/k}\frac{(p^{k+1}-1)^{1/k}}{q(p-1)^{1/k} }  < p  
\end{equation*}
As $q \geq 2$, it is enough to show the following 
\begin{equation*}
    \label{eqn:claim3_eqn2} 
  \bigg( \frac{1+2.4^k+2^k}{1+3.2^k} \bigg)^{1/k}\frac{p^{1/k}}{(p-1)^{1/k} }  < 2 . 
\end{equation*}
Taking $k$-th power on both sides and cross multiplying, this is equivalent to show the following: 
\begin{equation*}
    \label{eqn:claim3_eqn3} 
 p(1+2.4^k+2^k) < 2^k(1+3.2^k)(p-1).  
\end{equation*}
After some calculations, this boils down to show
\begin{equation}
    \label{eqn:claim3_eqn4} 
 \frac{2^{2k+1}+1+2^k} {2^{2k}-1} < p-1  .
\end{equation}
As $k \geq 1$, we have $2{k}+3 \leq 2^{2k}-1$ and thus $\displaystyle\frac{2^{2k+1}+1+2^k} {2^{2k}-1} < 3 $. The proof of \eqref{eqn:claim3_eqn4} is now complete as $p\geq 5.$

This completes the proof.
\end{proof}

We are now in a position to prove the main result of this paper. 

\section{Proof of Theorem \ref{thm:711generalization} } 

\begin{proof}[Proof of Theorem \ref{thm:711generalization}] 

Let $G$ be a finite non-cyclic group of order $n$ satisfying
\begin{equation}
    \psi_{k}(G)> \displaystyle \frac{1+3.2^k}{1+2.4^k+2^k}\psi_{k}(\mathbb{Z}_{n}). \label{eqn:Ggreater711Zn} 
\end{equation}
Let $q=p_1 < p_2 < \dots < p_t = p$ be the prime divisors of $n$. Moreover, we denote the corresponding Sylow subgroups of $\mathbb{Z}_n$ by $P_1, P_2, \dots, P_t.$ We proceed by induction on the size of $p$. 

By Lemma \ref{lem:lowerbound}, we have  $$\psi_{k}(G)>\displaystyle \frac{1+3.2^k}{1+2.4^k+2^k}\psi_{k}(\mathbb{Z}_{n})\geq \displaystyle  \frac{1+3.2^k}{1+2.4^k+2^k}\frac{q^k}{(p+1)^k}n^{k+1}.$$ This holds since $p$ is the largest prime and $q$ is the smallest prime dividing $n$ respectively. Thus there exists an element $x$ in $G$ such that $$o(x)> \bigg(\frac{1+3.2^k}{1+2.4^k+2^k} \bigg)^{1/k}\frac{q}{p+1}n.$$

{\underline {\it Case 1: $p=2$} } 

\medskip

We first assume that $p=2$. In this case $q$ must also equal 2. Then
\begin{equation*}
[G:\langle x \rangle ]<\frac{3}{2}\bigg(\frac{1+2.4^k+2^k}{1+3.2^k}\bigg)^{1/k}<\frac{3}{2} \bigg (\frac{3.4^k}{3.2^k} \bigg)^{1/k}=\frac{3}{2}.2=3 
\end{equation*}
and therefore $[G:\langle x \rangle ]=2$; Thus, $n\geq 4$. So,

\begin{eqnarray}
 \psi_{k}(G) & \leq & \psi_{k}(\mathbb{Z}_{n/2})+\bigg(\frac{n}{2} \bigg) \bigg(\frac{n}{2} \bigg)^k\nonumber\\
&= & \frac{\bigg(\displaystyle\frac{n}{2}\bigg)^{k+1}.2^k+2^k-1}{2^{k+1}-1}+\bigg(\frac{n}{2}\bigg)^{k+1}\nonumber\\
&= & n^{k+1} \bigg[\frac{1}{2(2^{k+1}-1)}+\frac{1}{2^{k+1}} \bigg]+\frac{2^k-1}{2^{k+1}-1}.\nonumber
\end{eqnarray}
The first line follows from Theorem \ref{thm:Amiri_kth_gen}. The second line follows by using Lemma \ref{lem:sumofkthpower} and the fact that as $n$ is a power of $2$, hence $\mathbb{Z}_{n/2}$ is a cyclic group of prime power. 

We also have  

\begin{eqnarray}
   &  &\frac{1+3.2^k}{1+2.4^k+2^k}\bigg(\frac{2^kn^{k+1}+2^k-1}{2^{k+1}-1}\bigg)-n^{k+1}\bigg[\frac{1}{2(2^{k+1}-1)}+\frac{1}{2^{k+1}} \bigg]-\frac{2^k-1}{2^{k+1}-1}\nonumber\\
    &= & n^{k+1} \bigg [\frac{2^k+3.4^k}{(2^{k+1}-1)(1+2.4^k+2^k)}-\frac{1}{2(2^{k+1}-1)}-\frac{1}{2^{k+1}} \bigg ]+\frac{2^k-1}{2^{k+1}-1} \bigg (\frac{1+3.2^k}{1+2.4^k+2^k}-1 \bigg )\nonumber\\
    &= &\frac {n^{k+1}\bigg({2^k.2^{k+1}+3.4^k2^{k+1}-2^k(1+2.4^k+2^k)-(2^{k+1}-1)(1+2.4^k+2^k)} \bigg) \nonumber  }{(2^{k+1}-1)2^{k+1}(1+2^k+2.4^k)} \\ & & +\frac{(2^k-1)(2.2^k-2.4^k)}{(2^{k+1}-1)(1+4^k.2+2^k)}\nonumber\\
    &= &\frac{n^{k+1}(1-2^{k+1}+4^k)}{(2^{k+1}-1)(1+4^k.2+2^k)2^{k+1}}+\frac{(2^k-1)(2.2^k-2.4^k)}{(2^{k+1}-1)(1+4^k.2+2^k)}\nonumber\\
    &= & \frac{n^{k+1}(1-2^{k+1}+4^k)-(2.4^k-2^{k+1})^2}{(2^{k+1}-1)(1+4^k.2+2^k)2^{k+1}}\nonumber\\
    &\geq & \frac{4^{k+1}(1-2^{k+1}+4^k)-(2.4^k-2^{k+1})^2}{(2^{k+1}-1)(1+4^k.2+2^k)2^{k+1}}( \mbox{Since, } \hspace{.4 mm} n\geq 4, n^{k+1}\geq 4^{k+1})\nonumber\\
    &= & 0.\nonumber
    \end{eqnarray}
    Thus, we have
    
    $$ \psi_k(G) \leq  \frac{1+3.2^k}{1+2.4^k+2^k}\bigg(\frac{2^kn^{k+1}+2^k-1}{2^{k+1}-1} \bigg)=\frac{1+3.2^k}{1+2.4^k+2^k}\psi_{k}(\mathbb{Z}_{n}),  $$ which is a contradiction to \eqref{eqn:Ggreater711Zn}.
    
    \bigskip
    
    {\underline {\it Case 2: $p=3$}}  
    
    \medskip
    
    We next assume $p=3$. As $p=3$, we must have $q=3$ or $q=2.$ If $q=3$, we have that $G$ is a $3$-group and hence by Lemma \ref{lem:k-upperbound}, we have  $$ \psi_{k}(G)<\frac{1}{2^k}\psi_{k}(\mathbb{Z}_{n})<\frac{1+3.2^k}{1+2.4^k+2^k}\psi_{k}(\mathbb{Z}_{n}),$$ which is again a contradiction. Thus, 
we may assume $q=2$, in which case we have $n=2^a3^b$ for some positive integers $a,b$. Here we have 

\begin{equation}
[G: \langle x \rangle ]<\frac{4}{2}\bigg (\frac{1+2.4^k+2^k}{1+3.2^k} \bigg)^{1/k} < 2\cdot\bigg (\frac{4^k}{2^k} \bigg )^{1/k}=4. 
\end{equation}
Hence, $[G: \langle x \rangle ]=2$ or $3$.
\medskip


\medskip

{\underline { \it Subcase 2a: $[G: \langle x \rangle ]=2$ } } 

\medskip

If $[G: \langle x \rangle ]=2$, then $\langle x \rangle $ contains a cyclic Sylow 3-subgroup $P$ of $G$ and since $\langle x \rangle \leq C_{G}(P)$, and hence, $P\vartriangleleft G$.

If there exist $y\in G\setminus \langle x \rangle $ with $[G: \langle y \rangle ]=2$ then $y\in C_{G}(P)$ and hence $P\leq Z(G)$. So, $G=P\times Q$, where $Q$ is non-cyclic Sylow 2-subgroup of $G$. Now from the case $p=2$ it follows that,
\begin{equation*}
\psi_{k}(G)=\displaystyle \psi_{k}(P)\psi_{k}(Q)\leq \psi_k(P)  \frac{1+3.2^k}{1+2.4^k+2^k}\psi_{k}(\mathbb{Z}_{|Q|})= \displaystyle \frac{1+3.2^k}{1+2.4^k+2^k}\psi_{k}(\mathbb{Z}_{n}),
\end{equation*} 
 which is a contradiction.
 So, now assume that $o(y)\leq \frac{n}{3}$ for all $y\in G\setminus \langle x \rangle $. Therefore, we have 
 $$ \psi_k(G) \leq   \psi_{k}(\mathbb{Z}_{n/2}) + \displaystyle \bigg(\frac{n}{2} \bigg) \bigg(\frac{n}{3} \bigg)^k.$$

By Lemma \ref{lem:ineqlemma_1}, we now have

$$ \psi_k(G) \leq  \frac{1+3.2^k}{1+2.4^k+2^k}\psi_{k}(\mathbb{Z}_{2^a})\psi_{k}(\mathbb{Z}_{3^b}),$$
which again contradicts \eqref{eqn:Ggreater711Zn}.

\medskip

{\underline { \it Subcase 2b: $[G: \langle x \rangle ]=3$ } }

\medskip

Next we consider the case when $p=3$ and $[G: \langle x \rangle ]=3$. Also as per the previous arguments we assume that there is no element of $G$ of order $\frac{n}{2}$ and hence, $o(y)\leq \frac{n}{3}$ for all $y\in G$. Thus, we obtain 
\begin{eqnarray}
   \psi_{k}(G) & \leq & \psi_{k}(\mathbb{Z}_{2^a})\psi_{k}(\mathbb{Z}_{3^{b-1}})+\bigg(\frac{2n}{3}\bigg)\bigg(\frac{n}{3}\bigg)^{k}\nonumber\\
  & \leq &   \frac{1+3.2^k}{1+2.4^k+2^k}\psi_{k}(\mathbb{Z}_{2^a})\psi_{k}(\mathbb{Z}_{3^b}).
 \nonumber
 \end{eqnarray}
 The second line follows by using Lemma \ref{lem:ineqlemma2}.
 This again contradicts \eqref{eqn:Ggreater711Zn}.

\medskip
 
 Hence, the theorem holds for $p\leq 3$ and now we assume that $p>3$. 
 
{\underline  {\it Case 3: $p>3.$}}  

\medskip 
 
By Lemma \ref{lem:lowerbound}, we have 
\begin{equation*}
    [G: \langle x \rangle ]< \displaystyle \bigg( \frac{1+2.4^k+2^k}{1+3.2^k} \bigg)^{1/k}\frac{(1+p+\dots+p^k)^{1/k}}{q} .
\end{equation*} 
As $p>3$, we have  $[G: \langle x \rangle ]<p$ by Lemma \ref{lem:ineq3lemma}.
Thus $\langle x \rangle $ contains a cyclic Sylow $p$-subgroup $P$ of $G$. Since, $\langle x \rangle \leq N_{G}(P)$, it follows that $P$ is cyclic and also $P\vartriangleleft G$. So using Lemma \ref{lem:analogue_cia}, we have 

\begin{equation*}
    \psi_{k}(P)\psi_{k}(G/P)\geq \psi_{k}(G)> \displaystyle \frac{1+3.2^k}{1+2.4^k+2^k}\psi_{k}(\mathbb{Z}_{p^r})\psi_{k}(\mathbb{Z}_{n/p^r}),
\end{equation*}
where, $|P|=p^r$. Now, $P\simeq \mathbb{Z}_{p^r}$ and by cancellation we obtain that,
\begin{equation*}
    \psi_{k}(G/P)> \displaystyle \frac{1+3.2^k}{1+2.4^k+2^k}\psi_{k}(\mathbb{Z}_{n/p^r}). 
\end{equation*}

Since $p$ is the largest prime dividing $n$, the maximal prime dividing $\displaystyle\frac{n}{p^r}$ is smaller than $p$. Using
Remark \ref{rem:semidirect_product_analog}, by our induction hypothesis $G/P$ is cyclic and $G= P\rtimes F$, with $F\cong G/P$ and $F\neq 1$. Now $n=|P||F|$, with $P$ and $F$ are cyclic and $(|P|,|F|)=1$. So,  $\psi_{k}(\mathbb{Z}_{n})=\psi_{k}(P)\psi_{k}(F)$.
\medskip

If $C_{F}(P)=F$, then $G=P\times F$ and $G$ is cyclic, which is a contradiction. So assume that $C_{F}(P)=Z<F$.
\begin{equation*}
    \psi_{k}(G)<\psi_{k}(P)\psi_{k}(F) \displaystyle \bigg(\frac{\psi_{k}(Z)}{\psi_{k}(F)}+\frac{|P|}{\psi_{k}(P)}\bigg)=\psi_{k}(\mathbb{Z}_{n})\bigg(\frac{\psi_{k}(Z)}{\psi_{k}(F)}+\frac{|P|}{\psi_{k}(P)}\bigg) .
\end{equation*}

Now since, $P$ is cyclic $p$-group and $p>3$, we have
\begin{eqnarray}
 & &\frac{|P|}{\psi_{k}(P)}\nonumber\\
 &=&\frac{|P|(p^{k+1}-1)}{|P|^{k+1}p^{k}(p-1)+(p^k-1)}\nonumber\\
 &<&\frac{|P|(p^k+p^{k-1}+\cdots+1)}{|P|^{k+1}p^{k}}\nonumber\\
 &\leq& \frac{(p^k+\cdots+1)}{p^{2k}}\nonumber\\
 &=&\frac{1}{p^k}+\frac{1}{p^{k+1}}+\cdots+\frac{1}{p^{2k}} \nonumber\\
 &\leq& \frac{1}{5^k}+\frac{1}{5^{k+1}}+\cdots+\frac{1}{5^{2k}}\nonumber\\
 &=&\frac{5^{k+1}-1}{4.5^{2k}}<\frac{5^{k+1}-1}{4.(5^{2k}-5^{k-1})}=\frac{5^{k+1}-1}{4.5^{k-1}(5^{k+1}-1)}=\frac{1}{4.5^{k-1}} .\label{eqn:last1}
\end{eqnarray}
Note that $Z$ is a proper subgroup of the cyclic group $F$ and $\psi_{k}(F)$ is a product of $\psi_{k}(S)$, with $S$ running over the Sylow subgroups of $F$. Moreover, $\psi_{k}(Z)$ is a similar product,  and at least one Sylow subgroup of $Z$, say Sylow $d$-subgroup $R_{Z}$, is properly contained in the Sylow $d$-subgroup $R_{F}$ of $F$ of order $d^s$ for some prime $d$.
By an argument, similar to the proof of \eqref{eqn:imp_lemma}, we have 
\begin{equation}
    \frac{\psi_{k}(Z)}{\psi_{k}(F)}\leq \frac{(d^{s-1})^{k+1}d^k(d-1)+(d^k-1)}{(d^s)^{k+1}d^k(d-1)+(d^k-1)}\leq \frac{d-1}{d^{k+1}-d^k}\leq \frac{1}{2^{k+1}-1}\bigg(\leq \frac{1+3.2^k}{1+2.4^k+2^k}\bigg)
    \label{eqn:last2} 
\end{equation}
 This is true since $d \geq 2.$
\medskip

Now combining these, we obtain,
\begin{eqnarray*}
    \psi_{k}(G)& \leq & \psi_{k}(\mathbb{Z}_{n}) \bigg(\frac{\psi_{k}(Z)}{\psi_{k}(F)}+\frac{|P|}{\psi_{k}(P)} \bigg) \\ &\leq & \psi_{k}(\mathbb{Z}_{n})\bigg(\frac{1}{4.5^{k-1}}+\frac{1}{2^{k+1}-1} \bigg) \\ 
    & \leq & \psi_{k}(\mathbb{Z}_{n})\bigg(\frac{1}{4.4^{k-1}}+\frac{1}{2^{k+1}-1} \bigg) \\
    & < &\psi_{k}(\mathbb{Z}_{n})\bigg(\frac{1+2^k}{1+2.4^k+2^k} +\frac{2.2^k}{1+2.4^k+2^k} \bigg) \\  & < & \psi_{k}(\mathbb{Z}_{n})\frac{1+3.2^k}{1+2.4^k+2^k},
\end{eqnarray*}
where the second line follows from \eqref{eqn:last1} and \eqref{eqn:last2}. The fourth line follows from some easy inequality computations and hence we omit that. Thus we obtain contradiction for all primes.

This completes the proof of the theorem. 
\end{proof} 

We finally give an example to show that the bound in Theorem \ref{thm:711generalization} is the best possible upper bound.

\begin{example}
Let $t$ be an odd positive integer and $n=4t.$ Then, by Lemma \ref{lem:analogue_cia}, we have $$\psi_k(\mathbb{Z}_n)=(1+2^k+2.4^k)\psi_k(\mathbb{Z}_t).$$ Moreover,  by Lemma \ref{lem:analogue_cia}, we also have $$\psi_k(\mathbb{Z}_{t} \times \mathbb{Z}_2 \times \mathbb{Z}_2)=(1+3.2^k)\psi_k(\mathbb{Z}_t).$$
Therefore, we have 
$$ \psi_k(\mathbb{Z}_{t} \times \mathbb{Z}_2 \times \mathbb{Z}_2)= \frac {1+3.2^k}{1+2^k+2.4^k} \psi_k(\mathbb{Z}_n).$$
\end{example}

\section*{Acknowledgements}
The first author thanks Prof.
Arvind Ayyer for his constant support and encouragement.
The first author acknowledges SERB-National Post Doctoral Fellowship (File No. PDF/2021/001899) during the preparation of this work and profusely thanks Science and Engineering Research Board, Govt. of India for this funding. The first author also acknowledges excellent working conditions in the Department of Mathematics, Indian Institute of Science.
\medskip

The second author acknowledges IIT Bombay TA fellowship  for supporting her financially during this work. Also, she would like to thank Buddhadev Hajra for some helpful discussions.


\begin{thebibliography}{20}	

\bibitem{amiri-communication-secondmax} S. M. Jafarian Amiri. \emph{Second maximum sum of element orders on finite nilpotent groups}, Communications in Algebra \textbf{41} (2013), 2055-2059.

\bibitem{amiri-algapplctn} H. Amiri, S. M. Jafarian Amiri. \emph{Sum of element orders on finite groups of the same order},
Journal of Algebra and its Applications \textbf{10} (2011), 187-190.



\bibitem{amiri-pureandApplied} S. M. Jafarian Amiri, M. Amiri. \textit{Second maximum sum of element orders on finite groups},
Journal of Pure and Applied Algebra \textbf{218} (2014), 531-539.

\bibitem{amiri-amiri-cia} S. M. Jafarian Amiri, M. Amiri. \textit{Sum of the Products of the Orders of Two Distinct
Elements in Finite Groups},
Communications in Algebra \textbf{42} (2014), 5319-5328

\bibitem{amiri-communication} H. Amiri, S. M. Jafarian Amiri, I. M. Isaacs. \emph{Sums of element orders in finite groups}, Communications in Algebra \textbf{37} (2009), 2978-2980. 

\bibitem{asad-joa} M. Baniasad Asad and B. Khosravi, \emph{A criterion for solvability of a finite group by the sum of element orders,} Journal of Algebra \textbf{516} (2018), 115-124.

\bibitem{chew-chin-lim} C. Y. Chew, A. Y. M. Chin, and C. S. Lim, \emph{A Recursive Formula for the Sum of Element Orders of Finite Abelian Groups,} Results in Mathematics \textbf{72} (2017), 1897-1905. 

\bibitem{Her-pureandApplied} M. Herzog, P. Longobardi and M. Maj, \emph{An exact upper bound for sums of element orders in non-cyclic finite groups,} Journal of Pure and Applied Algebras \textbf{222} (2018), 1628-1642.

\bibitem{Her-joa}M. Herzog, P. Longobardi and M. Maj, \emph{Two new criteria for solvability of finite groups in finite groups,} Journal of Algebra \textbf{511} (2018), 215-226.

\bibitem{Her-cia}M. Herzog, P. Longobardi and M. Maj, \emph{Sums of element orders in groups of order 2m with m odd,} Communications in Algebra \textbf{47} (2019), 2035-2948.

\bibitem{Her-jpaa}M. Herzog, P. Longobardi and M. Maj, \emph{The second maximal groups with respect to the sum of element orders,} \textbf{225} (2021), 106531.

\bibitem{Isac-ams}I. M. Isaaacs, \emph{Finite Group Theory,} Graduate Studies in mathematics, Vol. 92, American Mathematical Society, Providence, RI, 2008.

\bibitem{suvra} S. Saha, \emph{Sum of the Powers of the Orders of Elements in Finite Abelian Groups}, Advances in Group Theory and Applications, \textbf{13} (2022), 1-11.

\bibitem{Scott}W. R. Scott. \emph{Group Theory,} Prentice-Hall, Englewood Cliffs, NJ, 1964.


\bibitem{shen-et-al} R. Shen, G. Chen, C. Wu. \textit{On groups with the second largest value of the sum of element orders},
Communications in Algebra \textbf{43} (2015), 2618-2631.

\bibitem{tarnauceanu-israel} M. T\~{a}rn\~{a}uceanu, \textit{Detecting structural properties of finite groups by the sum of element orders}, Israel Journal of Mathmatics \textbf{238} (2020), 629–637. 


\end{thebibliography}
\end{document}